\documentclass[11pt]{amsart}
\author{K\"{u}bra Benl\.{i}}
\author{Paul Pollack}
\address{Department of Mathematics\\ University of Georgia\\ Athens, GA 30601}

\email{kubra.benli25@uga.edu}
\email{pollack@uga.edu}

\title[Small prime $k$th power residues]{Small prime $k$th power residues for $k=2,3,4$: \\A reciprocity laws approach}

\usepackage{amsmath,amssymb,amsthm,microtype,geometry,url}
\usepackage[hidelinks]{hyperref}
\DeclareMathAlphabet{\curly}{U}{rsfs}{m}{n}
\newtheorem{thm}{Theorem}

\newtheorem{prop}[thm]{Proposition}

\newtheorem*{thmA}{Theorem A}

\theoremstyle{remark}
\newtheorem*{rmk}{Remark}

\begin{document}
\renewcommand{\labelenumi}{(\roman{enumi})}
\renewcommand\phi\varphi
\def\A{\curly{A}}
\def\E{\curly{E}}
\def\F{\mathbb{F}}
\def\PV{\mathrm{PV}}
\def\Qq{\curly{Q}}
\def\N{\mathbb{N}}
\def\Q{\mathbb{Q}}
\def\Z{\mathbb{Z}}
\def\R{\mathbb{R}}
\def\C{\mathbb{C}}
\def\RSUM{\mathrm{RSUM}}
\def\WRSUM{\mathrm{WRSUM}}
\def\OO{\mathcal{O}^{*}}
\newcommand\qred{\mathrm{qred}}
\newcommand\SL{\mathrm{SL}}
\newcommand{\leg}[2]{\genfrac{(}{)}{}{}{#1}{#2}}
\newcommand\Hh{\mathbb{H}}
\newcommand\gcrd{\mathrm{gcrd}}
\newcommand\Mod{\mathrm{~Mod~}}

\begin{abstract} Nagell proved that for each prime $p\equiv 1\pmod{3}$, $p > 7$, there is a prime $q<2p^{1/2}$ that is a cubic residue modulo $p$. Here we show that for each fixed $\epsilon > 0$, and each prime $p\equiv 1\pmod{3}$ with $p > p_0(\epsilon)$, the number of prime cubic residues $q < p^{1/2+\epsilon}$ exceeds $p^{\epsilon/30}$. Our argument, like Nagell's, is rooted in the law of cubic reciprocity; somewhat surprisingly, character sum estimates play no role. We use the same method to establish related results about prime quadratic and biquadratic residues. For example, for all large primes $p$, there are more than $p^{1/9}$ prime quadratic residues $q<p$.
\end{abstract}
\maketitle
\section{Introduction}
For each prime $p$ and each integer $k\ge 2$, let $r_k(p)$ denote the smallest prime $k$th power residue modulo $p$. Clearly, any prime congruent to $1$ modulo $p$ is a $k$th power residue, and so $r_k(p)$ exists for all pairs $k, p$. Almost a full century ago, I.\,M. Vinogradov conjectured that $r_2(p) = O_{\epsilon}(p^{\epsilon})$ for each $\epsilon > 0$ \cite{vinogradov18}, and it is widely believed that the same is true for $r_k(p)$, for every fixed $k$. (The general conjecture is known under the assumption of the Generalized Riemann Hypothesis; see, e.g., a recent paper of Lamzouri, Li, and Soundararajan \cite[Theorem 1.4]{LLS15}, who present explicit upper bounds improving earlier estimates of Bach and Sorenson \cite{BS96}.) The jumping-off point for this note is an unconditional upper bound for $r_k(p)$ published by Elliott in 1974 \cite{elliott71}.
\begin{thmA} Fix an integer $k\ge 2$, and fix $\epsilon > 0$. For all large primes $p\equiv 1\pmod{k}$,
\begin{equation}\label{eq:elliottbound}r_k(p) <	 p^{\frac{k-1}{4}+\epsilon}.
\end{equation}
\end{thmA}
\noindent The restriction to primes $p\equiv 1\pmod{k}$ is a natural one, since the set of $k$th powers modulo $p$ coincides with the set of $\gcd(k,p-1)$th powers.

In nascent form, the method of proof of Theorem A goes back to Linnik and A.\,I. Vinogradov \cite{VL66}, who showed Theorem A when $k=2$. The key components are (1) Burgess's character sum bound, and (2) lower bounds on $|L(1,\chi)|$ for nonprincipal Dirichlet characters $\chi$ mod $p$ of order dividing $k$. Note that Theorem A is of interest only for fairly small values of $k$, as $\frac{k-1}{4}$ eventually exceeds the exponent in known versions of Linnik's theorem.

For odd values of $k$, Elliott observes (op. cit.) that the proof of Theorem A can be modified to give a slightly sharper upper bound on $r_k(p)$. (The improvement comes from our possessing better lower bounds on $|L(1,\chi)|$ for complex $\chi$ vis-\`a-vis real $\chi$.) As an example, he states that for primes $p\equiv 1\pmod{3}$,
\begin{equation*} r_3(p) \le c p^{\frac{1}{2}} \exp(c' \sqrt{\log p \cdot \log\log{p}}) \end{equation*}
for certain constants $c, c'>0$.  It does not seem to be widely known that Nagell published a still sharper upper bound for $r_3(p)$ already in 1952 \cite{nagell52}, namely
\begin{equation}\label{eq:nagellbound} r_3(p) < 2p^{\frac{1}{2}} \qquad\text{once $p>7$}.\end{equation}
Remarkably, Nagell's proof of \eqref{eq:nagellbound} is free of any trappings of analysis, relying instead on the algebraic theory of cubic residues developed by Gauss, Jacobi, and Eisenstein.

This note explores further consequences of Nagell's method for the distribution of prime $k$th power residues, for $k=2,3,4$.

To set the stage, observe that the $k=2$ case of Theorem A guarantees at least one prime quadratic residue below $p^{1/4+\epsilon}$. The second author showed in \cite{pollack17} that there are in fact \emph{many} prime quadratic residues below this bound: For any $\epsilon >0$ and $A>0$,
\begin{equation}\label{eq:manysmall2} \#\{\text{primes $q<p^{\frac{1}{4}+\epsilon}$: $q$ is a quadratic residue mod $p$}\} > (\log{p})^{A} \end{equation}
for all primes $p > p_0(\epsilon,A)$. Our first theorem is an analogous --- but in one sense superior --- result for prime cubic residues.

\begin{thm}\label{thm:main1} Let $\epsilon > 0$. For all primes $p \equiv 1\pmod{3}$, $p  > p_0(\epsilon)$, we have that
\[ \#\{\text{primes }q < p^{\frac{1}{2}+\epsilon}: q\text{ is a cubic residue mod $p$}\} > p^{\frac{1}{30}\epsilon}. \]
\end{thm}

\noindent This surpasses \eqref{eq:manysmall2} in that the number of power residues produced exceeds a certain power of $p$, not merely an arbitrary power of $\log{p}$. By contrast, the analytic method of \cite{pollack17} when applied to this problem gives only a weaker lower bound of
\[ p^{c \log\log\log p/\log\log p} \] for some absolute positive constant $c$.

Our second theorem concerns biquadratic (i.e., fourth power) residues. For an odd prime $q$, let $q^{\ast} = (-1)^{(q-1)/2} q$, so that $q^{\ast} = \pm q$ and $q^{\ast} \equiv 1\pmod{4}$.

\begin{thm}\label{thm:main2}  Let $\epsilon > 0$. For all primes $p\equiv 1\pmod{4}$, $p > p_0(\epsilon)$, we have that
\[ \#\{\text{primes }q < p^{\frac{1}{2}+\epsilon}: \text{$q^{\ast}$ is a biquadratic residue modulo $p$}\} > p^{\frac{1}{50}\epsilon}. \]
\end{thm}

\noindent If $p\equiv 1\pmod{8}$, then $-1$ is a biquadratic residue modulo $p$. Consequently, $q$ and $q^{\ast}$ are either both biquadratic residues or both biquadratic nonresidues. So in this case, Theorem \ref{thm:main2} implies a power-of-$p$ lower bound on the number of prime biquadratic residues $q < p^{1/2+\epsilon}$. In comparison, Theorem A only guarantees a single prime biquadratic residue below the significantly larger value $p^{3/4+\epsilon}$. (However, the bound of Theorem A applies also when $p\equiv 5\pmod{8}$.)

As our last application, we revisit the problem of showing that there are many small prime quadratic residues modulo an odd prime $p$. In \cite{pollack17}, ``small'' was taken to mean ``not much larger than $p^{1/4}$''. Here we show that for the more relaxed problem where $p^{1/4}$ is replaced by $p^{1/2}$ we can once again establish a power-of-$p$ lower bound.

\begin{thm}\label{thm:main3} Suppose that $0 < \epsilon \le \frac12$. For all primes $p > p_0(\epsilon)$, we have that
\[ \#\{\text{primes }q < p^{\frac{1}{2}+\epsilon}: \text{$q$ is a quadratic residue modulo $p$}\} > p^{\frac{1}{25}\epsilon}. \]
\end{thm}

As was the case for Theorem \ref{thm:main1}, the proofs of Theorems \ref{thm:main2} and \ref{thm:main3} are character-free. It would be interesting to investigate the possibility of  obtaining stronger results by injecting character sum estimates into the method.

\section{Many small prime cubic residues: Proof of Theorem \ref{thm:main1}}

The following consequence of the law of cubic reciprocity is due to Z.-H. Sun (see \cite[(1.6) and Corollary 2.1]{sun98}). Recall that for each prime $p\equiv 1\pmod{3}$, there are integers $L,M$, uniquely determined up to sign, with
\begin{equation}\label{eq:prep} 4p = L^2 + 27M^2. \end{equation}

\begin{prop}\label{prop:ZH} Let $p$ be a prime with $p \equiv 1\pmod{3}$, and let $L, M$ be integers satisfying $4p = L^2+27M^2$. Let $q$ be a prime, $q \ne 2, 3$, or $p$.  Then
\[ q \text{ is a cubic residue mod $p$} \Longleftrightarrow \frac{L}{3M}\equiv \frac{x^3-9x}{3(x^2-1)} \pmod{q} \quad\text{for some $x\in \Z$}.\]
In the case when one of the right-hand denominators is a multiple of $q$, the congruence is considered to hold when both denominators are multiples of $q$. In particular, if $q\mid M$, then $q$ is a cubic residue modulo $p$ {\rm (}take $x=1${\rm )}.
\end{prop}

\begin{rmk} Taking $x=0$ and $x=1$, we deduce from Proposition \ref{prop:ZH} that if a prime $q\ne 2,3$ divides $LM$, then $q$ is a cubic residue modulo $p$. In fact, the restriction to $q\ne 2,3$ is unnecessary. (See \cite[Chapter 7]{lemmermeyer00} or \cite[Chapter 2]{pollack09}  for details and background). When $p>7$, \eqref{eq:prep} implies that $|LM| > 1$. Now taking any prime $q$ dividing $LM$ produces a cubic residue with $q < 2p^{1/2}$. This was essentially Nagell's proof of \eqref{eq:nagellbound}.\end{rmk}

\begin{proof}[Proof of Theorem \ref{thm:main1}] Let $p$ be a large prime with $p\equiv 1\pmod{3}$, and write $4p = L^2+27M^2$ with $L, M > 0$. Let
\[ f_0(x) := M(x^3-9x) + L(x^2-1) \in \Z[x]. \]

As preparation for sieving the values of $f_0$, we record some observations on the number of roots of $f_0$ modulo primes $q$. Modulo $q=2$, there is always at least one root, since $f_0(1)=-8M$, and there are two roots whenever $L$ is even, since $f_0(0)=-L$. Modulo $q=3$, the polynomial $f_0$ has at most two roots, since $3\nmid f_0(0)$ (for if $3\mid L$, then $3^2 \mid L^2+27M^2=4p$). Now suppose that $q>3$. Since  $\gcd(L,M)^2 \mid L^2+27M^2 = 4p$, it must be that
\begin{equation}\label{eq:gcdLM}\gcd(L,M)=1\text{ or } 2.\end{equation}
Since $f_0$ has leading coefficient $M$ and constant term $-L$, \eqref{eq:gcdLM} implies that $f_0$ does not reduce to the zero polynomial mod $q$, and so $f_0$ has at most three roots modulo $q$. Collecting the results of this paragraph, we see in particular that $f_0$ has no fixed prime divisor except possibly $q=2$.

We sidestep the case when $f_0$ has $2$ as a fixed prime divisor by means of the following device. From \eqref{eq:gcdLM},  $2^5 \nmid \gcd(L,8M)$; hence, we may choose $n_0 \in \{0,1\}$ with $2^5 \nmid f_0(n_0)$. Let $e$ be the largest integer for which $2^e \mid f_0(n_0)$, so that $e \in \{0,1,2,3,4\}$. Put
\[ f(x) = \frac{1}{2^e} f_0(2^5 x + n_0). \]
Then $f(x) \in \Z[x]$ and all the values of $f$ at integer inputs are odd. Since $2^5$ is invertible modulo every odd prime $q$, the above results concerning $f_0$ imply that $f$ has at most two roots modulo $q=3$ and at most three roots modulo each prime $q>3$.

Now let
\[ \A= \{f(n): n\le p^{\epsilon/4}\}. \]
(Here and below, $n$ is understood to run only over positive integers.)
Since $f$ has no fixed prime divisors and at most three roots modulo every prime $q>3$, the fundamental lemma of the sieve shows that there is an absolute constant $\eta > 0$ such that
\[ \#\{n \le p^{\epsilon/4}, f(n)\text{ has no prime divisor less than $p^{\eta\epsilon/4}$}\} \gg_{\epsilon} p^{\epsilon/4}/(\log{p})^3 \]
provided only that $p$ is sufficiently large in terms of $\epsilon$. In fact, by the sieve of Diamond--Halberstam--Richert, we can (and will) take $\eta=1/7$. (The sieve we use is Theorem 9.1 on p. 104 of \cite{DH08}. The relevant numerological fact is that the sifting limit, $\beta_3$, is smaller than $7$; see Table 17.1, p. 227.) Put
\[ \E = \{n \le p^{\epsilon/4}: f(n)\text{ has no prime divisor less than $p^{\epsilon/28}$}\}, \]
and let
\[ \Qq = \{\text{primes $q$}: q\mid f(n)\text{ for some $n \in \E$}\}, \]
so that
\[ \#\E \gg_{\epsilon} p^{\epsilon/4}/(\log{p})^3 \qquad \text{and}\qquad \min \Qq \ge p^{\epsilon/28}. \]
It is easy to see that $f(n) > 1$ for all positive integers $n$. Thus,
\[ \sum_{n \in \E} 1 \le \sum_{n \in \E} \sum_{\substack{q \mid f(n)\\ q\text{ prime}}} 1. \]
Reversing the order of summation and using our lower bound on $\#\E$, we deduce that
\begin{equation}\label{eq:doublesumlower}  \sum_{q \in \Qq} \sum_{\substack{n \le p^{\epsilon/4} \\ q\mid f(n)}} 1 \gg_{\epsilon} p^{\epsilon/4}/(\log{p})^3 \end{equation}
for large $p$. On the other hand, for each $q \in \Qq$, the number of $n \le p^{\epsilon/4}$ for which $q\mid f(n)$ is at most $3p^{\epsilon/4}/q + O(1)$. Thus,
\begin{align*} \sum_{q \in \Qq} \sum_{\substack{n \le p^{\epsilon/4} \\ q\mid f(n)}} 1 &\le 3p^{\epsilon/4} \sum_{q \in \Qq}\frac{1}{q}  + O(\#\Qq) \\
&\le 3p^{\epsilon/4} \cdot p^{-\epsilon/28} \#\Qq + O(\#\Qq).\end{align*}
If we suppose that $\#\Qq \le p^{\epsilon/29}$, then this contradicts \eqref{eq:doublesumlower} (for large $p$). Hence,
\[ \#\Qq > p^{\epsilon/29}. \]

Take any $q \in \Qq$ with $q$ not dividing $6LM$. This non-divisibility condition  excludes only $O(\log{p})$ values of $q$, and so (for large $p$) there are still at least $p^{\epsilon/30}$ choices of $q$. We have that $q\mid f(n)$ for some $n \le p^{\epsilon/4}$, so that
if we set $m=2^5n+n_0$, then \[ M(m^3-9m) \equiv -L(m^2-1) \pmod{q}. \]
Thus,
\[ (m^3-9m) \equiv -\frac{L}{3M} (3(m^2-1)) \pmod{q}. \]
If $q\mid 3(m^2-1)$, then $m\equiv \pm 1\pmod{q}$, and so $m^3-9m \equiv \pm 8 \not\equiv 0 \pmod{q}$, contradicting the last displayed congruence. Thus, $3(m^2-1)$ is invertible modulo $q$, and
\[ \frac{L}{3M} \equiv -\frac{m^3-9m}{3(m^2-1)} \equiv \frac{(-m)^3 - 9(-m)}{3((-m)^2-1)} \pmod{q}. \]
By Proposition \ref{prop:ZH}, $q$ is a cube modulo $p$.

We will be finished if we show that each $q \in \Qq$ is smaller than $p^{1/2+\epsilon}$. But this is easy: $q$ divides a nonzero integer of the form $f_0(m)$ where $1 \le m \le 2^5 p^{\epsilon/4} + 1$. For every positive integer $m$,
\[ |f_0(m)| \le \max\{|L|,|M|\} (|m^3-9m| + |m^2-1|) \ll p^{\frac12} m^3.  \]
Thus, $|f_0(m)|$ and $q$ are both smaller than $p^{1/2+\epsilon}$ (for large $p$).
\end{proof}
\section{Biquadratic residues: Proof of Theorem \ref{thm:main2}}
For the proof of Theorem \ref{thm:main2}, we replace Proposition \ref{prop:ZH} with the following corollary to the biquadratic reciprocity law. Recall that each prime $p\equiv 1\pmod{4}$ admits a representation $p=L^2+4M^2$, with the integers $L,M$ uniquely determined up to sign.

\begin{prop}\label{prop:ZH2} Let $p$ be a prime with $p \equiv 1\pmod{4}$, and let $L, M$ be integers satisfying $p = L^2+4M^2$. Let $q$ be an odd prime, $q \ne p$.  Then
\begin{multline*} q^{\ast} \text{ is a biquadratic residue mod $p$} \Longleftrightarrow \\ \frac{L}{2M}\equiv \frac{x^4-6x^2+1}{4(x^3-x)} \pmod{q} \quad\text{for some $x\in \Z$}.\end{multline*}
In the case when one of the right-hand denominators is a multiple of $q$, the congruence is considered to hold when both denominators are multiples of $q$. In particular, if $q\mid M$, then $q^{\ast}$ is a biquadratic residue modulo $p$ {\rm (}take $x=0${\rm )}.
\end{prop}

\noindent Proposition \ref{prop:ZH2} is again due to Sun (compare with Theorem 2.2 and Corollary 3.2 of \cite{sun01}).

\begin{proof}[Proof of Theorem \ref{thm:main2}] The proof closely parallels that of Theorem \ref{thm:main1}. This time, we let $L,M$ be positive integers with $L^2+4M^2=p$, and we put
\[ f_0(x) = M(x^4-6x^2+1)+ 2L(x^3-x) \in \Z[x].	\]
If $2\mid M$, then $q=2$ is clearly a fixed prime divisor of $f_0$. Noting that $n^3-n$ is always a multiple of $3$, we see that when $3\mid M$ the prime $q=3$ is also a fixed divisor of $f_0$. Now suppose that $q \ge 5$. Since $\gcd(L,M)^2 \mid L^2 + 4M^2 = p$, it is clear that $\gcd(L,M)=1$. The constant term of $f_0$ is $M$ while the $x^3$-coefficient is $2L$; since $q\ge 5$ and $\gcd(L,M)=1$, at least one of $M$ and $2L$ is not a multiple of $q$. Hence, $f_0$ does not reduce to the zero polynomial modulo $q$, and so $f_0$ has at most four roots modulo $q$.

As before, fixed prime divisors can be avoided by restricting the the allowed substitutions for $x$ to a suitable arithmetic progression. Let $2^{e}$ be the highest power of $2$ dividing $f_0(0)=M$ and let $3^{e'}$ be the highest power of $3$ dividing $M$. If $e\ge 3$, then $2^2 \parallel f_0(2)= -7M + 12L$ (since in that case, $2\nmid L$, and so $2^2\parallel 12L$). Similarly, if $e' \ge 2$, then $3\parallel f_0(2)$. Set
\[ m =
\begin{cases}
0 &\text{if $2^3 \nmid M$},\\
2 &\text{if $2^3 \mid M$},
\end{cases}
\quad\text{and}\quad
m' =
\begin{cases}
0 &\text{if $3^2 \nmid M$},\\
2 &\text{if $3^2 \mid M$}.
\end{cases}
\]
Let $n_0$ be a positive integer solution to the simultaneous congruences
\[ n_0 \equiv m \pmod{2^3}, \quad n_0\equiv m' \pmod{3^2} \]
with $n_0 \le 2^3 \cdot 3^2$.
If $v, v'$ are defined by the conditions that $2^{v} \parallel f_0(n_0)$ and $3^{v'} \parallel f_0(n_0)$, we have $v \in \{0,1,2\}$ and $v' \in \{0,1\}$. Put
\[ f(x) = \frac{1}{2^{v} 3^{v'}} f_0(2^3 3^2 x + n_0). \]
Then $f(x) \in \Z[x]$ and all the values assumed by $f$ are coprime to $6$. Since $2^3 3^2$ is invertible modulo every prime $q\ge 5$, our earlier discussion of $f_0$ implies that $f$ has at most four roots modulo all these $q$.

Applying the sieve in the same manner as in the proof of Theorem \ref{thm:main1} shows that the number of primes $q$ dividing $f(n)$ for some $n \le p^{\epsilon/5}$ is at least $p^{\epsilon/46}$, for all large $p$. (We use the entry for $\beta_4$ in Table 17.1 of \cite{DH08} this time, since now $f$ can have up to four roots modulo a prime number $q$. Note that $\beta_4 < 9.1$, and $5\cdot 9.1<46$.) An argument entirely analogous to one seen above, now using Proposition \ref{prop:ZH2} in place of Proposition \ref{prop:ZH}, shows that $q^{\ast}$ is a biquadratic residue modulo $p$ for all such $q$ not dividing $2LM$. Since this last condition eliminates only $O(\log{p})$ primes, the number of remaining values of $q$ is at least $p^{\epsilon/50}$ (for large $p$). Finally, it is easy to see that $|f(n)| < p^{1/2+\epsilon}$ for all $n \le p^{\epsilon/5}$, so that each $q < p^{1/2+\epsilon}$.
\end{proof}

\section{Small quadratic residues redux: Proof of Theorem \ref{thm:main3}}

\begin{proof}[Proof of Theorem \ref{thm:main3}] Suppose first that $p\equiv 1\pmod{4}$. Let $r = \lfloor\sqrt{p}\rfloor$, and let $f(x) = (x+r)^2-p$. Then $f$ has no fixed prime divisor, and $f$ has at most two roots modulo every prime.  Applying the DHR sieve in a now familiar way, we find that
\[\#\{\text{odd primes }q: q \mid f(n) \text{ for some $n \le p^{0.95\epsilon}$}\} > p^{2\epsilon/9} \]
for sufficiently large $p$. (The relevant sifting limit this time, $\beta_2$, is $\approx 4.27$, and $0.95/4.27 > 2/9$.) For $n \le p^{0.95\epsilon}$, the integer $f(n)$ is positive and smaller than $p^{1/2+\epsilon}$; thus, each prime $q$ counted above is smaller than $p^{1/2+\epsilon}$. Moreover, for any of these $q$, the shape of $f$ makes it obvious that $p$ is a square modulo $q$. By quadratic reciprocity, $q$ is a square modulo $p$. This completes the proof of the theorem, in slightly stronger form, when $p\equiv 1\pmod{4}$.

We have to work harder when $p\equiv 3\pmod{4}$. Consider the reduced positive definite binary quadratic forms $ax^2+bxy+cy^2$ of discriminant $p^{\ast}=-p$;  note that such forms are necessarily primitive.  Let $h = h(-p)$ be the corresponding class number. A simple counting argument shows that (at least) one of our $h$ forms has $a \gg h/\log(2h)$. To see this, note that $a$ determines $b$ in $O(d(a))$ ways, since $b$ satisfies $b^2\equiv -p\pmod{a}$ and $|b| \le a$. Moreover, $c$ is determined by $a$ and $b$ via $b^2-4ac=-p$. Hence, if $A$ is the largest value of $a$ above, then
\[ h \ll \sum_{m \le A} d(m) \ll A \log{2A}; \]
consequently, $A \gg h/\log(2h)$, as claimed.

By Siegel's theorem, we have $h > p^{1/2-\epsilon/3}$ for all large $p$. Hence, one of the above forms has $a > p^{1/2-\epsilon/2}$. Since $|b| \le a \le c$,
\[ ac = \frac{b^2+p}{4} \le \frac{ac+p}{4}, \]
so that
\[ ac \le \frac{p}{3}. \]
Thus,
\[ c \le \frac{p}{3a} < p^{1/2+\epsilon/2}, \]
and (using $a^2 \le ac$)
\[ a < p^{1/2}. \]

The rest of the argument follows the usual lines, with $f_0(x) = ax^2 + bx + c$. Since $\gcd(a,b,c)=1$, the reduction of $f_0$ is nonzero modulo every prime $q$, and so $f_0$ has at most two roots modulo $q$. In particular, $q=2$ is the only possible fixed prime divisor. Notice that if $2^2 \mid c=f_0(0)$ and $2^2 \mid 4a+2b+c = f_0(2)$, then $2 \mid b$; since $\gcd(a,b,c)=1$, this forces $a$ to be odd, so that $a+b+c=f_0(1)$ is also odd. So we can choose an $n_0 \in \{0,1,2\}$ for which $2^2\nmid f_0(n_0)$. Define $e$ by the condition that $2^e \parallel f_0(n_0)$, so that $e=0$ or $1$, and let
\[ f(x) = \frac{1}{2^e} f_0(2^2 x + n_0). \]
Then $f(x)\in\Z[x]$, $f$ assumes only odd values, and $f$ has at most two roots modulo every prime. Application of the sieve shows that
\[ \#\{\text{odd primes }q: q \mid f(n)\text{ for some $n \le p^{\epsilon/5}$}\} > p^{\epsilon/25} \]
for large $p$. (We use the crude estimate $5\cdot 4.27 < 25$.) Each of these $q$ divides a nonzero integer $f_0(m)$ for some positive integer $m < p^{0.21\epsilon}$ (say). Thus,
\[ q \le |f_0(m)| \le a m^2 + |b| m + c \le c(m^2 + m + 1) \le p^{1/2+0.5\epsilon} \cdot p^{0.43\epsilon} < p^{1/2+\epsilon}. \]
Moreover, since the discriminant of $f_0$ is $p^{\ast}$, we may conclude that $p^{\ast}$ is a square modulo $q$. By quadratic reciprocity, $q$ is a square modulo $p$. \end{proof}

\begin{rmk} Our Theorems 1 and 2 are effective in the technical sense; given $\epsilon > 0$, there is no theoretical obstacle to computing the value of $p_0(\epsilon)$. The same is true for Theorem 3 in those cases when $p\equiv 1\pmod{4}$; however, when $p\equiv 3\pmod{4}$, the invocation of Siegel's theorem means that we have no way of estimating the required lower bound on $p$. It seems interesting to note that for the simpler problem of counting prime quadratic residues smaller than $p$ (the specific case $\epsilon=1/2$ of Theorem 3), effectivity is easily restored. One simply applies our sieve argument to $f_0(x) = x^2+x+\frac{1-p^{\ast}}{4}$. In this way, one can show that for all primes $p$ larger than an effectively computable absolute constant, there are more than $p^{1/9}$ prime quadratic residues $q < p$. Here $9$ could be replaced with any number larger than $2\cdot 4.27$. (In addition to being effective, the exponent $1/9$ is better --- i.e., larger --- than the one that comes directly out of the proof of Theorem 3.)
\end{rmk}
	
\section*{Acknowledgements} Our interest in this area was sparked by a \texttt{MathOverflow} posting of ``GH from MO'' \cite{52393}. We thank GH for the inspiration.

\providecommand{\bysame}{\leavevmode\hbox to3em{\hrulefill}\thinspace}
\providecommand{\MR}{\relax\ifhmode\unskip\space\fi MR }
\providecommand{\MRhref}[2]{%
  \href{http://www.ams.org/mathscinet-getitem?mr=#1}{#2}
}
\providecommand{\href}[2]{#2}

\end{document}